\documentclass[12pt]{amsart}
\usepackage[english]{babel}
\usepackage[all]{xy}
 \usepackage[latin1]{inputenc}
  \usepackage[T1]{fontenc}
  \usepackage{amsmath,amssymb}
  \usepackage{amsmath}
  \usepackage{amsthm}
\usepackage{amssymb}
\usepackage{mathrsfs}
\usepackage{enumerate}
\usepackage{tikz}
\usetikzlibrary{shapes.geometric}
\usepackage{physics}
\usepackage[notcite, final, notref]{showkeys}
\usepackage{dsfont}
\newtheorem{theorem}{Theorem}[section]
\newtheorem{problem}[theorem]{Problem}
\newtheorem{lemma}[theorem]{Lemma}
\newtheorem{proposition}[theorem]{Proposition}
\newtheorem{question}[theorem]{Question}
\newtheorem{corollary}[theorem]{Corollary}
\newtheorem{definition}[theorem]{Definition}
\newtheorem{notation}[theorem]{Notation}
\theoremstyle{definition}
\newtheorem{remark}[theorem]{Remark}

\usepackage{tikz}
\usetikzlibrary{arrows,shapes,positioning,shadows,trees,matrix}
\usepackage{xcolor}

\title[]{Finitely additive functions in measure theory and applications}

\author[D. Alpay]{Daniel Alpay}
\address{(DA) Schmid College of Science and Technology \\
Chapman University\\
One University Drive
Orange, California 92866\\
USA}
\email{alpay@chapman.edu}

\author[P. Jorgensen]{Palle Jorgensen}
\address{(PJ)
Department of Mathematics, 14 MLH\\The University
of Iowa, Iowa City, Iowa 52242-1419\\ USA}
\email{palle-jorgensen@uiowa.edu}

\begin{document}
\maketitle

\begin{abstract}
  In this paper, we consider, and make precise, a certain extension of the Radon-Nikodym derivative operator, to functions which are additive, but  not necessarily sigma-additive, on a subset of a given sigma-algebra. We give applications to probability theory; in particular, to the study of $\mu$-Brownian motion,
  to stochastic calculus via generalized  It\^o-integrals, and their adjoints (in the form of generalized stochastic derivatives), to systems of transition probability operators indexed by families of measures $\mu$, and to adjoints of composition operators.
\end{abstract}

\noindent AMS Classification: Primary: 47B32, 60G20, 60G15, 60H05, 60J60. Secondary: 46E22.\\

\noindent{\em Keywords:}
Hilbert space, reproducing kernels, probability space, Gaussian fields, transforms, covariance, It\^o integration, It\^o calculus,
generalized Brownian motion.

\tableofcontents

\section{Introduction}
\setcounter{equation}{0}
Motivated by a diverse set of applications, in the context of probability theory, we present here a general result (Theorem \ref{mainth}) on finitely additive functions. We
demonstrate its implications for the study of a stochastic calculus based on generalized  It\^o-integrals, and generalized derivatives, for a prescribed systems of sigma-finite positive measures, see especially Theorems \ref{th5-1}, \ref{567} and \ref{th5-4} below. \smallskip

To provide motivation, consider the following example. Let $f\in\mathbf L^2(\mathbb R,\mathcal B,dx)\setminus\mathbf L^1(\mathbb R,\mathcal B,dx)$ (the classical Lebesgue spaces of the real line).
The function
\begin{equation}
  f(A)=\int_A f(x)dx
\end{equation}
is additive on the algebra of finite length (measurable) sets, but will not be sigma-additive since $f$ is not summable. The question
we address more generally is the following:

\begin{question} Given a measure space $(X,\mathcal F,\mu)$, where $\mu$ is sigma-finite,
define $\mathcal F^{\rm fin}$ to
be the family of sets of finite measure for $\mu$. 
The question is to give an intrinsic characterization of the functions of the form
\begin{equation}
    \label{MAMA}
M(A)=\int_A f(x)\mu(dx),\quad A\in \mathcal F^{\rm fin},
\end{equation}
where $f\in\mathbf L^2(X,\mathcal F,\mu)$.
\end{question}

We note that $\mathcal F^{\rm fin}$ generates $\mathcal F$ (in the sense that $\mathcal F$ is
the smallest sigma-algebra containing $\mathcal F^{\rm fin}$). This problem
was first suggested, and discussed briefly, by S.D. Chatterji in \cite{MR0448536}, in dealing with cases when the derivative need not be
assumed summable. The 
motivation in that paper was the theory of convergence of martingales. Our motivation comes from the
theory of composition operators. Consider a measure space $(X,\mathcal F)$, an endomorphism $\sigma$ of $X$, and a sigma-finite
measure $\mu$ which is $\sigma$-invariant:
\begin{equation}
  \label{sig-inv}
\mu\circ\sigma^{-1}=\mu,
\end{equation}
meaning that 
\[
\mu(A)=\mu(\sigma^{-1}(A)),\quad A\in\mathcal F.
\]
It follows that the composition map $S$:
\begin{equation}
  \label{comp-s}
f\mapsto f\circ \sigma
\end{equation}
is an isometry from $\mathbf L^p(X,\mathcal F,\mu)$ into itself for $p\in[1,\infty)$. As we will illustrate in Section \ref{S-*}
in the case $p=2$, the computation of the adjoint $S^*$ involves the extension of the Radon-Nikodym theorem considered here.\smallskip

The problem addressed in the present work is further motivated by a key idea from It\^o calculus; in particular, on the fact that the It\^o-integral is based on $L^2$
theory. Following for example \cite{MR1366434,MR2053326,MR3199302} one notes that the It\^o-integral takes the form of an isometry between
the respective $L^2$-spaces. This is true also for the extension of It\^o's theory which is based on a version of Brownian motion, or the
Wiener process, $W^{(\mu)}$  governed by an arbitrary sigma-finite measure $\mu$, as opposed to the more familiar case of Lebesgue measure;
see Section \ref{sec-4} below. Denoting by $V_\mu$  the It\^o-isometry calculated from  $W^{(\mu)}$, it is then natural to view the adjoint
operator $V_\mu^*$ (now a co-isometry) as a generalized derivative operator. But this entails a separate  $L^2$ approach for such a
generalized derivative; so one not relying on more familiar notions of Radon-Nikodym derivatives for $\mu$. Here we present such a
theory, accompanied with applications which in turn entail a new stochastic analysis based on families of sigma-finite measures, and
their associated It\^o-calculus.\smallskip

{\bf Overview.}The paper consists of five sections besides the introduction.  In Section \ref{sec-2} we review some properties of the reproducing kernel
Hilbert space with reproducing kernel $\mu(A\cap B)$, where $\mu$ is a sigma-finite measure and $A,B$ run in ${\mathcal F}^{\rm fin}$. The main result of the paper is proved in Section \ref{sec-3}. In the last three sections  we consider applications, to the $\mu$-Brownian motion, transition probability systems and adjoint of composition operators respectively.



\section{The reproducing kernel Hilbert space $\mathfrak H(\mu)$}
\setcounter{equation}{0}
\label{sec-2}
In preparation to Theorem \ref{mainth}, we recall the definition of the reproducing kernel Hilbert space associated to a sigma-finite measure and some of
its properties. We refer to  \cite{aj_opu, aj_jotp} for further details. With the notation of the introduction, we have
the following result, see \cite{aj_jotp}.

\begin{theorem}
  \label{MAMA-TH}  The function $K^{(\mu)}(A,B)=\mu(A\cap B)$ is positive definite on $\mathcal F^{\rm fin}$ and the associated reproducing kernel Hilbert space
  consists of the functions of the form \eqref{MAMA},
where $f\in\mathbf L^2(X,\mathcal F,\mu)$, with norm $\|M\|=\|f\|_2$ (where $\|f\|_2$ denotes the norm of $f$ in $\mathbf L^2(X,\mathcal F,\mu)$).
\end{theorem}

\begin{definition} {\rm 
We denote by $\mathfrak H(\mu)$ the reproducing kernel Hilbert space with reproducing kernel $\mu(A\cap B)$, $A,B\in\mathcal F^{\rm fin}$.}
  \end{definition}
It follows from \eqref{MAMA} that $M(A)=0$ when $\mu(A)=0$. When $\mu(X)<\infty$, $M$ is a signed measure and the function $f$ is equal to
the Radon-Nikodym derivative $\frac{{\rm d}M}{{\rm d}\mu}$. This latter interpretation fails when $\mu$ is not finite, and we will see in
Theorem \ref{mainth} another characterization of $M$.\\

The following lemma will be used in the proof of Theorem \ref{mainth}.

\begin{lemma}
  \label{Iowa-city-0502-2022}
  Let $B_1,\ldots, B_J\in\mathcal F^{\rm fin}$, and let $c_1,\ldots, c_J\in\mathbb C$.
 The sum $ \sum_{j=1}^J b_j1_{B_j}$ can be rewritten in the form $\sum_{n=1}^{N} a_n1_{A_n}$ where the sets
$A_n$ are pairwise disjoint.
\end{lemma}

\begin{proof} The proof is a repeated use of the formula $A=(A\setminus B)\cup (A\cap B)$. We use induction. For $N=2$, one writes
  \[
    \begin{split}
      b_11_{B_1}+b_21_{B_2}&=c_1\left(1_{B_1\setminus B_2}+1_{B_1\cap B_2}\right)+c_2\left(1_{B_2\setminus B_1}+1_{B_2\cap B_1}\right)\\
      &=c_11_{B_1\setminus B_2}+c_21_{B_2\cap B_1}+(c_1+c_2)1_{B_1\cap B_2}.
  \end{split}
\]
Assuming the result true at rank $J$ we have
\[
\begin{split}
  \sum_{j=1}^{J+1} b_j1_{B_j}&=\sum_{m=1}^{J} b_j1_{B_j}+b_{J+1}1_{B_{J+1}}\\
  &=\sum_{n=1}^Na_n1_{A_n}+b_{J+1}1_{B_{J+1}}\\
  &=\sum_{n=1}^Na_n1_{A_n\setminus B_{J+1}}+\sum_{n=1}^Na_n1_{A_n\cap B_{J+1}}+b_{J+1}1_{B_{J+1}\setminus\cup_{n=1}^N A_n}+b_{J+1}1_{B_{J+1}\cap \left(\cup_{n=1}^N A_n\right)}\\
  &=\sum_{n=1}^Na_n1_{A_n\setminus B_{J+1}}+\sum_{n=1}^Na_n1_{A_n\cap B_{J+1}}+b_{J+1}1_{B_{J+1}\setminus\cup_{n=1}^N A_n}+\\
  &\hspace{5mm}+\sum_{n=1}^Nb_{J+1}1_{B_{J+1}\cap A_n}\\
  &=\sum_{n=1}^Na_n1_{A_n\setminus B_{J+1}}+\sum_{n=1}^N(a_n+b_{J+1})1_{A_n\cap B_{J+1}}+b_{J+1}1_{B_{J+1}\setminus\cup_{n=1}^N A_n}.
  \end{split}
\]  
\end{proof}

The following result is a special case of the characterization of the elements of a reproducing kernel Hilbert space; see e.g. \cite{aron}.

\begin{theorem}
  \label{ineq-90}
  A function $M$ defined on $\mathcal F^{\rm fin}$ belongs to $\mathfrak H(\mu)$ with norm less or equal to $\sqrt{C}$ if and only if the kernel
  \[
\mu(A\cap B)-\frac{1}{C}M(A)\overline{M(B)}
\]
is positive definite on $\mathcal F^{\rm fin}$.
  \end{theorem}
\section{The main result}
\setcounter{equation}{0}
\label{sec-3}
Let as above $\mathcal F$ be a sigma-algebra on a
set $X$, let $\mu$ be a sigma-finite measure, and let $\mathcal F^{\rm fin}$ denote the sets of finite measure for $\mu$.
In \cite{MR0448536} the following problem was considered:

\begin{problem} {\rm
Given is a complex-valued additive function $M$
on $\mathcal F^{\rm fin}$ which is absolutely continuous with respect to $\mu$ in the sense that:
\[
\forall A\in \mathcal B^{\rm fin},\quad \mu(A)=0\,\,\Longrightarrow M(A)=0.
\]
The problem was to characterize $M$ in alternative ways.}
\end{problem}

\begin{theorem}
  \label{mainth}
  The following three conditions are equivalent:\smallskip

  $(1)$ There exists $h\in\mathbf L^2(X,\mathcal F,\mu)$ such that
  \begin{equation}
    \label{KF-deriv}
M(A)=\int_Ah(x)\mu(dx),\quad A\in\mathcal F^{\rm fin}.
   \end{equation} 
   $(2)$ There exists a constant $C<\infty$ such that, for every $N\in\mathbb N$ and every family $A_1,\ldots, A_N$ of pairwise disjoint
   elements in $\mathcal F^{\rm fin}$, it holds that
   \begin{equation}
     \label{new-cond}
\sum_{n=1}^N\frac{|M(A_n)|^2}{\mu(A_n)}\le C.
\end{equation}
$(3)$ $M\in\mathfrak H(\mu)$.
\end{theorem}  

The proof of the equivalence between $(1)$ and $(2)$ is outlined in \cite[Theorem 6, p. 17]{MR0448536} for $p\in(1,\infty)$. Here, for $p=2$, we prove this equivalence and add the equivalence with $(3)$.
The difference with the arguments in \cite[pp. 17-18]{MR0448536} is that we do not prove directly that $(2)$ implies $(1)$, but that $(2)$ implies $(3)$, and then prove that $(3)$ implies $(1)$.
The case of general $p\in(1,\infty)$ is recalled below (see Theorem \ref{mainth2}).
\begin{proof}[Proof of Theorem \ref{mainth}]
  We first show that $(1)$ implies $(2)$. Assume $(1)$ holds.
  Let $A\in\mathcal F^{\rm fin}$. By Cauchy-Schwarz inequality we have
  \[
    \begin{split}
      |M(A)|^2&=|\int_X 1_A(x)(1_A(x)h(x))\mu(dx)|^2\\
     & \le \left(\int_X1^2_A(x)\mu(dx)\right)\left(\int_X1^2_A(x)|h(x)|^2\mu(dx)\right)=\mu(A)\int_A|h(x)|^2\mu(dx).
      \end{split}
    \]
    Thus, for $A_1,\ldots, A_N$ pairwise disjoint elements of $\mathcal F^{\rm fin}$, we can write:
    \[
      \sum_{n=1}^N\frac{|M(A_n)|^2}{\mu(A_n)}\le \sum_{n=1}^N\int_{A_n}|h(x)|^2dx\le\int_X|h(x)|^2dx,
    \]
    so that $(2)$ holds. Assuming $(2)$, in order to prove $(3)$ we will show that the kernel
    \[
\mu(A\cap B)-\frac{M(A)\overline{M(B)}}{C}
\]
is positive definite on $\mathcal F^{\rm fin}$ for $C=\|h\|_2^2$.
Let $b_1,\ldots, b_J$
be complex numbers and $B_1,\ldots, B_J$ be in $\mathcal F^{\rm fin}$. Using Lemma \ref{Iowa-city-0502-2022} we rewrite
\[
\sum_{j=1}^Jb_j1_{B_j}=\sum_{n=1}^Na_nA_n
\]
where $a_1,\ldots, a_N\in\mathbb C$ and now the $A_n$ are pairwise disjoint. By Cauchy-Schwarz inequality,
\[
  \begin{split}
    \frac{1}{C}\left|\sum_{n=1}^N a_nM(A_n)\right|^2
        &=\frac{1}{C}\left|\sum_{n=1}^N \frac{M(A_n)}{\sqrt{\mu(A_n)}}a_n\sqrt{\mu(A_n)}\right|^2\\
        &\le      \frac{1}{C}\left(\sum_{n=1}^N\frac{M(A_n)|^2}{\mu(A_n)}\right)\left(\sum_{n=1}^N|a_n|^2\mu(A_n)\right)\\
        &\le\sum_{\ell, n=1}^Na_\ell\overline{a_n}\mu(A_\ell\cap A_n)
    \end{split}
  \]
  since the $A_n$ are pairwise disjoint. By Theorem \ref{ineq-90}, the function $A\mapsto M(A)$ belongs to $\mathfrak H(\mu)$ with norm at least
  $C$. The fact that $(3)$ implies $(1)$ forms the content of Theorem \ref{MAMA-TH}.
  \end{proof}

   \begin{definition}{\rm
       We will use the notation $\nabla_\mu$ for        the map which to $M\in\mathfrak H(\mu)$ associates $h$ as
    in \eqref{KF-deriv}, and call it the Krein-Feller derivative. We can therefore rewrite \eqref{KF-deriv} as
    \begin{equation}
      \label{KF-deriv-2}
      M(A)=\int_A(\nabla_{\mu}M)(x)\mu(dx),\quad A\in\mathcal F^{\rm fin}.
      \end{equation}}
    \end{definition}

    Our motivation for this terminology comes  from analysis on fractals, where various variants of the operator of differentiation by $\mu$ appear in the
    theory of the Krein-Feller diffusion. The generator of the Krein-Feller diffusion is a variant of  $L:= \frac{d^2}{ dx d\mu}$. See for instance
    \cite{MR2213587,MR0091556,MR2787628,MR4295177,MR4048458}. See also \cite{Dmk}.\\
    
    As a consequence of Theorem \ref{mainth} we have:

    \begin{corollary}
      The map $\nabla_\mu$ is unitary from $\mathfrak H(\mu)$ onto $\mathbf L^2(X,\mathcal F,\mu)$, with adjoint given by
      \begin{equation}
        (\nabla_\mu^*g)(A)=\int_Ag(x)\mu(dx),\quad A\in\mathcal F^{\rm fin},\,\, g\in \mathbf L^2(X,\mathcal F,\mu).
        \label{adjoint-nabla}
        \end{equation}
      \end{corollary}

      \begin{proof}
        By Theorem \ref{MAMA-TH} the map \eqref{adjoint-nabla} is unitary from $\mathbf L^2(X,\mathcal F,\mu)$ onto $\mathfrak H(\mu)$. Denoting
        temporarily this map by $I_\mu$, we take $M\in\mathfrak H(\mu)$ of the form $M(A)=\int_Ah(x)\mu(dx)$ (with $h\in \mathbf L^2(X,\mathcal F,\mu)$).
The fact that $I_\mu=\nabla_\mu^*$ follows from
        \[
          \begin{split}
            \langle I_\mu g, M\rangle_{\mathfrak H(\mu)}=\langle g,h\rangle_{\mu}=\langle g,\nabla_\mu M\rangle_{\mu},
            \quad g\in \mathbf L^2(X,\mathcal F,\mu).
\end{split}
\]

        \end{proof}
    
        In the general case where $p\in(1,\infty)$ the two first items in Theorem \ref{mainth} are still equivalent, as we now prove; we follow, with a bit more details, the arguments in
        \cite[pp. 17-18]{MR0448536}. One could replace the third
      condition in Theorem \ref{mainth} by introducing pairs of spaces in duality (see \cite{Aronszajn60,aro3,a-nlsa} for the latter), but this will not be done here.

  \begin{theorem} (the case $p\in(1,\infty)$)
\label{mainth2}
  The following are equivalent:\smallskip

  $(1)$ There exists $h\in\mathbf L^p(X,\mathcal F,\mu)$ such that
  \begin{equation}
    \label{KF-deriv-p}
M(A)=\int_Ah(x)\mu(dx),\quad A\in\mathcal F^{\rm fin}.
   \end{equation} 
   $(2)$ There exists a constant $C<\infty$ such that, for every $N\in\mathbb N$ and every family $A_1,\ldots, A_N$ of pairwise disjoint
   elements in $\mathcal F^{\rm fin}$, it holds that
   \begin{equation}
     \label{new-cond-p}
\sum_{n=1}^N\frac{|M(A_n)|^p}{(\mu(A_n))^{p-1}}\le C.
\end{equation}
  \end{theorem}

  \begin{proof}
    The proof follows the proof of Theorem \ref{mainth}, but now uses H\"{o}lder inequality. Assume first that $(1)$ is in force. We have for $A\in\mathcal F^{\rm fin}$:
    \[
      \begin{split}
        |M(A)|&\le\int_X1_A(x)1_A(x)|h(x)|\mu(dx)\\
        &\le\left(\int_X(1_A(x))^q\mu(dx)\right)^{1/q}\left(\int_X(1_A(x))^p|h(x)|^p\mu(dx)\right)^{1/p}\\
        &\le (\mu(A))^{1/q}\left(\int_A|h(x)|^p\mu(dx)\right)^{1/p}.
        \end{split}
      \]
      Since $p/q=p-1$ we have
      \[
        |M(A)|^p\le(\mu(A))^{p-1}\int_A|h(x)|^p\mu(dx).
      \]
      Hence, for $A_1,\ldots, A_N$ pairwise disjoint elements of $\mathcal F^{\rm fin}$ we have:
      \[
\sum_{n=1}^N\frac{|M(A_n)|^p}{(\mu(A_n))^{p-1}}=\int_{\cup_{n=1}^NA_n}|h(x)|^p\mu(dx)\le\int_X|h(x)|^p\mu(dx).
        \]
    Assume now that \eqref{new-cond-p} is in force and define a map
   on the linear span of the functions $1_{A}$, $A\in\mathcal F^{\rm fin}$, by
    \[
      \varphi(f)      =\sum_{n=1}^Nc_nM(A_n),
    \]
    where $f=\sum_{n=1}^Nc_n1_{A_n}$, the sets $A_1,\ldots, A_N$ being moreover pairwise disjoint. Then, by H\"{o}lder's inequality (and with $1/p+1/q=1$)
    \[
      \begin{split}
        |\varphi(f)|&\le\sum_{n=1}^N|c_n|\mu(A_n)^{1/q}\frac{|M(A_n)|}{(\mu(A_n))^{1/q}}\\
          &= \left(\sum_{n=1}^N|c_n|^q\mu(A_n)\right)^{1/q}\left(\sum_{n=1}^N\frac{|M(A_n)|^p}{(\mu(A_n))^{p-1}}\right)^{1/p}\\
              &\le C \left(\sum_{n=1}^N|c_n|^q\mu(A_n)\right)^{1/q}\\
            &=C\left(\int_X|f(x)|^q\mu(dx)\right)^{1/q}
          \end{split}
        \]
        since $p/q=p-1$ Hence $\varphi$ extends to a continuous functional on $\mathbf L^q(X,\mathcal F,\mu)$.
        The claim follows then from Riesz theorem.    \end{proof}

      As a corollary we have:

      \begin{theorem} Let $p\in(1,\infty)$.
        A function $f$ defined on the real line is of the form $f(x)=\int_0^xg(u)du$ where $g\in\mathbf L^p(\mathbb R,\mathcal B,du)$ (the Borel sets and the
        Lebesgue measure) if and only if there exists $C>0$
        \[
\sum_{n=1}^N \frac{|f(x_{n+1})-f(x_n)|^p}{|x_{n+1}-x_n|^{p-1}}\le C
\]
for all $N\in\mathbb N$ and any ordered set of real points $x_1<x_1<\cdots< x_N$.
\end{theorem}
\section{Application to the $\mu$-Brownian motion}
\setcounter{equation}{0}
\label{sec-4}
Given a measure space $(X,\mathcal F)$ and a sigma-finite measure $\mu$ on $X$,  one introduces in a natural way three Hilbert spaces:\\
$(1)$ The Hilbert space $\mathbf L^2(X,\mathcal F,\mu)$.\\
$(2)$ The reproducing kernel Hilbert space $\mathfrak H(\mu)$ of functions defined on $\mathcal F^{\rm fin}$ with reproducing kernel
$K_\mu(A,B)=\mu(A\cap B).$\\
$(3)$ A probability space $ \mathbf L^2(\Omega,\mathcal C,\mathbb P) $ in which is constructed the
$\mu$-Brownian motion $W^{(\mu)}$ with covariance function $K^{(\mu)}(A,B)$,
\begin{equation}
  \label{4-1}
\mathbb E(W^{(\mu)}_AW^{(\mu)}_B)=\mu(A\cap B),\quad A,B\in\mathcal F^{\rm fin},
  \end{equation}
and  associated It\^o-type stochastic integrals
\begin{equation}
  V_\mu(f)=\int_X h(x)dW_{x}^{(\mu)},\quad h\in\mathbf L^2(X,\mathcal F,\mu).
\label{I-mu}
\end{equation}

\begin{remark}{\rm
In fact, $V_\mu$ is isometric into any $L^2$ probability space for which \eqref{4-1} is satisfied.
}
\end{remark}
We recall the following (see \cite{aj_jotp}):
\begin{proposition}
  \label{prop-4-1}
  The map $V_\mu$ is an isometry from $\mathbf L^2(X,\mathcal F,\mu)$ into $ \mathbf L^2(\Omega,\mathcal C,\mathbb P) $.
  \end{proposition}
In the  commutative diagram 
\begin{center}

  \usetikzlibrary{arrows.meta}

\begin{tikzpicture}
[>={Classical TikZ Rightarrow[length=1.2mm]}]
\def\a{2.5} \def\b{2.5}
\path[nodes={inner sep=1pt}]
(-\a,0) node (A) {$\mathbf L^2(X,\mathcal F,\mu)$}      
(\a,0) node (B) {$ \mathbf L^2(\Omega,\mathcal C,\mathbb P) $}
(0,-\b) node[align=center] (C) {$\mathfrak H(\mu)$\\[-1.3mm]\rotatebox{90}{}\\[-1.5mm]};
\begin{scope}[nodes={midway,scale=.75}]
\draw[->] (A)--(B) node[above]{$V_\mu$};
\draw[<-] (A)--(C.120) node[left]{$\nabla_\mu$};
\draw[<-] (C.60)--(B) node[right]{$T_\mu$};
\end{scope}
\end{tikzpicture}   
\end{center}

the maps $V_\mu$ and $\nabla_\mu$ were defined above and the map $T_\mu$ is defined by
\[
V_\mu\nabla_\mu T_\mu=I_{\mathbf L^2(\Omega,\mathcal C,\mathbb P)},
  \]
  so that
  \begin{equation}
    \label{TV}
    T_\mu=\nabla_{\mu}^{*}V_\mu^*
  \end{equation}
  since $V_\mu$ is an isometry and $\nabla_\mu$ is unitary.

\begin{proposition}
  The map $T_\mu$ and its adjoint are given by
\begin{eqnarray}
  (T_\mu\psi)(A)&=&\mathbb E\left(\psi W^{(\mu)}_A\right),\\
  T_\mu^*M&=&\int_Xh(x)dW^{(\mu)}_x
  \end{eqnarray}
where $h\in \mathbf L^2(X,\mathcal F,\mu)$ and $M\in\mathfrak H(\mu)$ is defined by $M(A)=\int_Ah(x)\mu(dx)$.
\end{proposition}

\begin{proof}
  Using \eqref{TV} we have for $\psi\in\mathbf L^2(\Omega,\mathcal C,\mathbb P) $ and $A\in\mathcal F^{\rm fin}$:
  \[
    \begin{split}
      (T_\mu\psi)(A)&=    \int_A (V_\mu^*\psi)(x)\mu(dx)\\
      &=\langle V_\mu^*\psi,1_A\rangle_\mu\\
      &=\langle \psi, W^{(\mu)}_A\rangle_{\mathbb P}\\
      &=\mathbb E(\psi W_A^{(\mu)}).
    \end{split}
  \]
  Furthermore,
  \begin{equation}
    T_\mu^*M=V_\mu\nabla_\mu M=V_\mu h=\int_Xh(x)dW_x^{(\mu)}.
    \label{babylone}
    \end{equation}
  \end{proof}

  \begin{corollary}
    It holds that
    \begin{equation}
      T_\mu^*(K^{(\mu)}(\cdot, A))=W^{(\mu)}_A,\quad A\in\mathcal F^{\rm fin}.
      \end{equation}
  \end{corollary}

  \begin{proof}
This is a special case of \eqref{babylone}, with $M(A)=K^{(\mu)}(\cdot, A)$, corresponding to $h(x)=1_A(x)$.
    \end{proof}
  
 In  Theorem \ref{mainth} we have defined a new kind of derivative, that allows us to give a precise characterization of $V_\mu^*$, which has the flavor of a derivative
  operator, and is presented in the following corollary:
  
  \begin{corollary}
    The adjoint of the map $V_\mu$ is given by
    \begin{equation}
      \label{v-mu-star}
  V_\mu^*\psi=\nabla_\mu\mathbb E (\psi W^{(\mu)}_\cdot),
\end{equation}
which we will also write as
      \begin{equation}
        \label{v-mu-star-2}
        (V_\mu^*\psi)(x)=\frac{d\mathbb E(\psi W^{(\mu)}_\cdot)}{d\mu}(x),
      \end{equation}
the precise meaning of this expression being given in terms of the operator $\nabla_\mu$.
    \end{corollary}


  \begin{remark}{\rm It follows from the above and from Proposition \ref{prop-4-1} that
      \begin{equation}
        \label{op}
\int_X(\nabla_\mu^*\psi)(x)dW_x^{(\mu)}=\mathbb E(\psi|\mathcal C_\mu)
        \end{equation}
holds for all $\psi\in\mathbf L^2(\Omega,\mathcal C,\mathbb P)$. Hence \eqref{op} justifies calling $\psi\mapsto V_\mu^*\psi$ an It\^o derivative.}
    \end{remark}

    We now interpret some of the previous results in terms of a conditional expectation in the underlying probability space
    $\mathbf L^2(\Omega,\mathcal C,\mathbb P)$.

    \begin{proposition} 
      Let $\mathcal C_\mu$ denote the sigma-algebra generated by the random variables $W^{(\mu)}_A$, $A\in\mathcal F^{\rm fin}$.
      Then,
      \begin{equation}
        T_\mu^*T_\mu=\mathbb E(\cdot|\mathcal C_\mu).
      \end{equation}
      \label{propcmu}
      \end{proposition}

      \begin{proof}
        In view of \eqref{TV}, and since $\nabla_\mu$ is unitary, it is enough to show that
        \begin{equation}
          \label{dindon}
          V_\mu V_\mu^*\psi=\mathbb E(\psi|\mathcal C_\mu),\quad \psi\in \mathbf L^2(\Omega,\mathcal C,\mathbb P),
        \end{equation}
        i.e. the $\mathcal C_\mu$ conditional expectation.         This in turn is equivalent to verify that
        \begin{equation}
          \langle V_\mu V_\mu^*\psi,   \int_Xf(x)dW^{(\mu)}_x\rangle_{\mathbb P}=    \mathbb E(\psi(\overline{\int_Xf(x)dW^{(\mu)}_x})),\quad
        \forall f\in\mathbf L^2(X,\mathcal F,\mathbb P),
        \label{lalala}
      \end{equation}
      i.e.
      \begin{equation}
          \langle V_\mu V_\mu^*\psi,   V_\mu f\rangle_{\mathbb P}=    \mathbb E_P(\psi(\overline{\int_Xf(x)dW^{(\mu)}_x})).
          \label{lalila}
          \end{equation}
        We can restrict $f$ to be of the form $1_A$ with $A\in\mathcal F^{\rm fin}$. Since $V_\mu$ is an isometry  \eqref{lalila} becomes equivalent to
        \[
\int_A (V_\mu^*\psi)(x)\mu(dx)=\mathbb E (\psi W_A).
\]
Applying $\nabla_\mu$ on both sides we get
\[
V_\mu^*\psi=\nabla_\mu(\mathbb E (\psi W_A)),
\]
which is nothing but $V_\mu^*=\nabla_\mu T_\mu$, which holds in view of \eqref{TV}.
        \end{proof}

      \begin{remark}{\rm We set $Q_\mu=V_\mu V_\mu^*$. Note that
          \begin{equation}
            Q_\mu V_\mu=V_\mu.
          \end{equation}
          \label{danton}
        }
      \end{remark}
      
      \section{Transition-probability systems}
      \setcounter{equation}{0}
      There are diverse approaches to the following general question: Given some stochastic data, then find an appropriate probability space
      $(\Omega, \mathcal C, {\mathbb P})$ that
realizes what is needed for the particular data at hand. Below we make this precise and we offer a brief outline with citations, especially \cite{MR0282379,MR0402840}.
By probability space $(\Omega, \mathcal C, {\mathbb P})$ we mean a triple consisting of a sample set $\Omega$, a sigma-algebra of events, and a probability
measure ${\mathbb P}$
defined on $\mathcal C$. Of the following four approaches to the problem, for our present purpose, number $(ii)$ is best suited. The list of four is: $(i)$
via Kolmogorov consistency,  $(ii)$ via Gaussian Hilbert space, $(iii)$ via transition kernels, and with the use of $(iv)$ generalized Gelfand triples.
While for many purposes, the Kolmogorov consistency construction $(i)$ is more constructive; here $(ii)$ is better, i.e., via $(ii)$ we obtain a probability space
$(\Omega, \mathcal C, {\mathbb P})$ from the following Gaussian Hilbert space construction:  Starting with a Hilbert space $\mathcal H$, we select a realization of the vectors
$h\in\mathcal H$ as a canonical Gaussian process $W_h$.  Hence, the realization of a Gaussian Hilbert space in some $(\Omega, \mathcal C, {\mathbb P})$ has
its associated covariance kernel equal to the inner product from $\mathcal H$.
Here we may use construction $(ii)$ on the canonical and universal Hilbert space in the sense of Nelson and Schwartz. We recall that this universal Hilbert space is a
Hilbert space of specific equivalence classes of pairs.\smallskip

{\bf Discussion of (ii):} For our present applications, we begin with a given generalized measure space $(X, \mathcal F)$, where $\mathcal F$ is a prescribed sigma-algebra. Consider systems of positive measures $(\mu)$ . We note that the positive measures $\mu$ will be based on the same  $\mathcal F$, but of
course the ring  $\mathcal F^{\rm fin}(\mu) $ will depend on $\mu$.
Of the books covering Gaussian Hilbert space and their applications, we stress \cite{MR0282379}. Summary of details for the construction leading from $\mathcal H$ to
$(\Omega, \mathcal C, {\mathbb P})$ has: $(a)$ We let $\mathcal H = \mathfrak H_{(X, \mathcal F)}$ to be the universal Hilbert space in the sense of Nelson and Schwartz, and then:
$(b)$, via an associated system of It\^o-isometries, we pass to a choice of a  ``universal'' $(\Omega, \mathcal C, {\mathbb P})$ probability space.
Specializing to two $\mu$-Brownian motions, say  $W^{(\mu_i)}, i= 1,2$, for $(\Omega, \mathcal C, {\mathbb P})$ they will be independent if and only if the two measures $\mu_i$
are mutually singular. \smallskip

In brief summary, the remaining two approaches are as follows: $(iii)$ Fix a system $(\mu)$, create an associated system of transition
kernels, and then construct $(\Omega, \mathcal F, {\mathbb P})$ from the combined Markov kernels. For other purposes, of course, we have $(iv)$
Gelfand triple constructions, see e.g., \cite{ajl_jotp}. For completeness we recall the construction of the universal Hilbert space; see \cite{Nelson_flows}.

\begin{definition}
  {\rm Given a fixed measure space $(X,\mathcal F)$, the associated universal Hilbert space $\mathfrak H_{(X,\mathcal F)}$ consists of equivalence classes of  pairs $(f, \mu)$, where $\mu$ is a positive measure on $(X,\mathcal F)$ and $f\in\mathbf L^2(X,\mathcal F)$. One says that
    $(f_1,\mu_1)\sim (f_2,\mu_2)$ if there exists a positive measure $\nu$ on $(X,\mathcal F)$  such that $\mu_1<<\nu$ and $\mu_2<<\nu$ and
    \begin{equation}
      \label{sch}
f_1\sqrt{\frac{{\rm d}\mu_1}{{\rm d}\nu}}=f_2\sqrt{\frac{{\rm d}\mu_2}{{\rm d}\nu}}, \quad \nu \,\,a.e.  
\end{equation}
    }
  \end{definition}
  It is known (see \cite{Nelson_flows}) that \eqref{sch} is indeed an equivalence relation, and an equivalence class for this relation will be denoted by $f\sqrt{\mu}$. The set of equivalence classes
  endowed with the norm $\|f\sqrt{\mu}\|^2_{\mathfrak H_{(X,\mathcal F)}}=\int_X|f(x)|^2\mu(dx)$ where $(f,\mu)\in f\sqrt{\mu}$ is a Hilbert space. We denote by $\mathbf L^2(\Omega,\mathcal C,{\mathbb P})$ the associated universal probability space, constructed as follows: One considers an orthonormal basis $(e_a)_{a\in A}$  of $\mathfrak H_{(X,\mathcal F)}$ and build
  \[
    \Omega=\prod_{a\in A}( \mathbb R,\frac{1}{\sqrt{2\pi}}e^{-\frac{x^2}{2}}dx),
    \]
endowed with the cylinder algebra; see \cite[pp. 38-39]{Neveu68}.
  \begin{theorem}
The It\^o integrals pass through the equivalence relation, meaning that the map
    \[
      f\sqrt{\mu}\,\,\,\mapsto\,\,\, V_\mu f\in\mathbf L^2(\Omega,\mathcal C,{\mathbb P})
    \]
    is a well defined isometry from the universal Hilbert space into the associated universal probability space.
    \end{theorem}

    \begin{proof}
      It holds that
      \[
        \int_Xf_1(x)dW^{(\mu_1)}_x=\int_Xf_1(x)\sqrt{\frac{{\rm d}\mu_1}{{\rm d}\nu}}(x)dW^{(\nu)}_x
          =\int_Xf_2(x)\sqrt{\frac{{\rm d}\mu_2}{{\rm d}\nu}}(x)dW^{(\nu)}_x=\int_Xf_2(x)dW^{(\mu_2)}_x.        
        \]
      \end{proof}

      As a corollary we have:

      \begin{corollary}
        Given two sigma-finite measures $\mu_1$ and $\mu_2$ on $X$. The following are equivalent:\\
        $(1)$ $\mu_1$ and $\mu_2$ are mutually singular.\\
        $(2)$ The corresponding $\mu$-Brownian motions $W^{(\mu_1)}$ and $W^{(\mu_2)}$ are independent.
        \end{corollary}

\begin{theorem} (transition probability systems)
  \label{th5-1}
  Using notation \eqref{v-mu-star-2} we have:
  \begin{eqnarray}
    V_{\mu_1}^*(W^{(\mu_2)}_B)(x)&=&\frac{{\rm d}\mathbb E(W^{(\mu_1)}_\cdot W_B^{(\mu_2)})}{{\rm d}\mu_1}(x)\\
    V_{\mu_2}^*(W^{(\mu_1)}_A)(y)&=&\frac{{\rm d}\mathbb E(W^{(\mu_2)}_\cdot W_A^{(\mu_1)})}{{\rm d}\mu_2}(y).
    \end{eqnarray}
  \end{theorem}

  \begin{proof}
The first formula is a special case of \eqref{v-mu-star} with $\mu=\mu_1$ and $\psi=W^{(\mu_2)}_B$. The second formula interchanges the indices $1$ and $2$.
  \end{proof}

      \begin{tikzpicture}
\node[] (M) at (0,0) {$(X,\mu_1$)};
\node[] (B) at (5,0) {$\vspace{-2mm}(X,\mu_2)$};
\node[] (C) at (0,5) {};
\node[] (D) at (5,5) {};
\node[]  (E) at(0,3)[right] {$x$};
\node[draw] (F) at(5,4) [left]{$B$};
%
\draw[] (B) -- (D);
\draw (C) -- (M);
\draw[->] (E)--node[above=0.05cm]{$P(x,\cdot)$} (F);
%
%
%
\node[] (B1) at (15,0) {$\vspace{-2mm}(X,\mu_2)$};
\node[] (M1) at (10,0) {$(X,\mu_1$)};
\node[] (C1) at (10,5) {};
\node[] (D1) at (15,5) {};
\node[draw]  (E1) at(10,3)[right] {$A$};
\node[] (F1) at(15,4) [left]{$y$};
%
\draw[] (B1) -- (D1);
\draw (C1) -- (M1);
\draw[->] (F1)--node[above=0.05cm]{$Q(y,\cdot)$} (E1);
      \end{tikzpicture}
  
  \begin{notation} We set
    \begin{eqnarray}
      \label{w34}
\frac{{\rm d}\mathbb E(W^{(\mu_1)}_\cdot W_B^{(\mu_2)})}{{\rm d}\mu_1}(x)&=&P(x,B)\\
\label{w45}
      \frac{{\rm d}\mathbb E(W^{(\mu_2)}_\cdot W_A^{(\mu_1)})}{{\rm d}\mu_2}(y)&=&Q(y,A).
                                                                                 \end{eqnarray}
                                                                               \end{notation}

                                                                               \begin{theorem}
                                                                                 \label{567}
    It holds that
                                                                                 \begin{eqnarray}
                                                                                   \label{lundi-matin}
\left(V_{\mu_1}^*V_{\mu_2}f_2\right)(x)&=&\int_XP(x,dy)f_2(y),\quad f_2\in\mathbf L^2(X,\mathcal F,\mu_2)\\
                                                                                   \label{mardi}
\left(V_{\mu_2}^*V_{\mu_1}f_1\right)(y)&=&\int_XQ(y,dx)f_1(x),\quad f_1\in\mathbf L^2(X,\mathcal F,\mu_1).
                                                                                                                              \end{eqnarray}
                                                                                                                            \end{theorem}

                                                                                                                            \begin{proof}
It is enough to prove these formulas with $f_2=1_B$ in the first case and $f_1=1_A$ in the second case.
\eqref{lundi-matin} reduces then to   \eqref{w34}. Formula \eqref{mardi} follows in a similar way from \eqref{w45}.                                                                  \end{proof}

\begin{theorem}
\label{th5-4}
  (reversibility)
  In the above notations, the following holds:
  \begin{equation}
\int_AP(x,B)\mu_1(dx)=\int_BQ(y,A)\mu_2(dy)=\mathbb E(W_A^{(\mu_1)}W_B^{(\mu_2)}),\quad A,B\in{\mathcal F}^{\rm fin}.
    \end{equation}
  \end{theorem}

  \begin{proof}
This is just an application of \eqref{KF-deriv-2} to the functions $M$ and $N$ defined by $M(A)=N(B)=\mathbb E(W_A^{(\mu_1)}W_B^{(\mu_2)})$.
  \end{proof}

We now study a related Markov property, and begin with a definition. In the statement, and as in Proposition \ref{propcmu}, we denote by
$\mathcal C_\mu$ the sigma-algebra generated by the $\mu$-Brownian motion.  

  \begin{definition}{\rm
Let $\mu_2$ and $\mu_3$ be two sigma-finite positive measures on $(X,\mathcal F)$. We say that the transition $\mu_2\longrightarrow\mu_3$ is
{\sl anticipating} if
\begin{equation}
  \label{ineqC}
  \mathcal C_{\mu_3}\subset\mathcal C_{\mu_2}.
\end{equation}
    }
  \end{definition}
  With $Q_\mu$ as in Remark \ref{danton} we can rewrite \eqref{ineqC} in terms of orthogonal projections as
  \begin{equation}
    \label{st-michel}
    Q_{\mu_3}  =  Q_{\mu_2}    Q_{\mu_3}.
\end{equation}
  We set for $f_2\in\mathbf L^2(X,\mathcal F,\mu_2)$
  \[
(V_1^*V_2f_2)(x)=\int_XP_{1\mapsto 2}(x,dy)f_2(y),\quad \mu_1 \,a.e.
  \]
 see \eqref{lundi-matin}, and similarly for other indices.
  
  \begin{theorem}
    Given three positive sigma-finite measures $\mu_1,\mu_2,\mu_3$ on $X$ the following are equivalent:\\
$(1)$ The transition equation
    \begin{equation}
      \label{transition}
      P_{1\mapsto 3}(x,B)=\int_X  P_{1\mapsto 2}(x,dy)P_{2\mapsto 3}(y,B),\quad B\in\mathcal F^{\rm fin}
    \end{equation}
    holds.\\
$(2)$ $\mu_3$ is anticipating $\mu_2$.
    \end{theorem}

\begin{proof}
The Markov property \eqref{transition} follows from the operator identity
  \begin{equation}
    \label{transition-456}
    (V_1^*V_2)(V_2^*V_3)=V_1^*V_3,
  \end{equation}
  which we rewrite as
  \[
  V_1^*Q_2Q_3V_3=V_1^*Q_3V_3.
  \]
It is immediate that the converse implication holds as well.
  \end{proof}

\section{Adjoint of the composition map}
\setcounter{equation}{0}
\label{S-*}
We now go back to the example presented in the introduction. The setting consists of a measure space $(X,\mathcal F)$, an endomorphism $\sigma$ of $X$, and
a sigma-finite measure $\mu$ which is $\sigma$-invariant (see \eqref{sig-inv}). We compute the adjoint of the composition map \eqref{comp-s} using the Krein-Feller
derivative.

\begin{theorem}
The adjoint of the operator $S$ is given by the Krein-Feller derivative of the map
\begin{equation}
  A\mapsto M_g(A)=\int_X 1_A(\sigma(x))g(x)\mu(dx).
  \label{new-g}
\end{equation}
\end{theorem}
    
\begin{proof}
  Let $f\in\mathbf L^2(X,\mathcal F,\mu)$. The composition map $S$ is isometric and therefore there exists $h\in\mathbf L^2(X,\mathcal F,\mu)$ such that
      \begin{equation}
        \int_X f(\sigma(x))\overline{g(x)}\mu(dx)=\int_Xf(x)\overline{h(x)}\mu(dx).
        \label{cid-05-02-2022}
      \end{equation}
      Taking conjugate and setting $f=1_A$ with $A\in\mathcal F^{\rm fin}$ we have:
      \begin{equation}
\int_X 1_A(\sigma(x))g(x)\mu(dx)=\int_Ah(x)\mu(dx).
\end{equation}
If follows from Theorem \ref{mainth} that the function
\eqref{new-g} belongs to $\mathfrak H(\mu)$ and has Krein-Feller derivative $h$.
      \end{proof}

        \begin{remark}{\rm
            In an informal way one sometimes uses the notation $(g\mu)\circ\sigma^{-1}$ for the map $M_g$, and the adjoint is given by the formula
            \begin{equation}
              S^*g=\nabla_\mu \left((g\mu)\circ\sigma^{-1}\right).
            \end{equation}
            }
          \end{remark}
      
      \begin{remark}{\rm We now check directly that the map \eqref{new-g} satisfies \eqref{new-cond}. To that purpose, let $N\in\mathbb N$ and
          let $A_1,\ldots, A_N$ be non-intersecting elements of $\mathcal F^{\rm fin}$. We note that
          \[
            1_A(\sigma(x))=1_{\sigma^{-1}(A)}(x)
          \]
          and rewrite $M_g$ as $M_g(A)=\int_X1_{\sigma^{-1}(A)}(1_{\sigma^{-1}(A)}g(x))\mu(dx)$.
          We then obtain from Cauchy-Schwarz inequality
          \[
            \begin{split}
              |M_g(A_n)|^2&\le\mu(\sigma^{-1}(A_n))\int_{\sigma^{-1}(A_n)}|g(x)|^2\mu(dx)\\
              &=\mu(A_n)\int_{\sigma^{-1}(A_n)}|g(x)|^2\mu(dx)
              \end{split}
                        \]
                        since $\mu$ is $\sigma$-invariant. Furthermore the sets $\sigma^{-1}(A_n)$ are disjoints since the $A_n$ are pairwise disjoints.
                        Hence
                        \[
                          \begin{split}
                            \sum_{n=1}^N\frac{|M_g(A_n)|^2}{\mu(A_n)}&\le\int_{\sigma^{-1}(A_n)}\sum_{n=1}^N\frac{\mu(A_n)}{\mu(A_n)}|g(x)|^2\mu(dx)\\
                            &\le\int_{\cup_{n=1}^N\sigma^{-1}(A_n)}|g(x)|^2\mu(dx)\\
                            &\le \|g\|_2.
                          \end{split}
                          \]
        }
      \end{remark}

\section*{Acknowledgments}
Daniel Alpay thanks the Foster G. and Mary McGaw Professorship in Mathematical Sciences, which supported this research.

\def\cprime{$'$} \def\cprime{$'$} \def\cprime{$'$}
  \def\lfhook#1{\setbox0=\hbox{#1}{\ooalign{\hidewidth
  \lower1.5ex\hbox{'}\hidewidth\crcr\unhbox0}}} \def\cprime{$'$}
  \def\cprime{$'$} \def\cprime{$'$} \def\cprime{$'$} \def\cprime{$'$}
  \def\cprime{$'$}

\end{document}